\documentclass[reqno]{amsart}
\usepackage{amssymb}
\usepackage[dvips]{epsfig}
\usepackage{graphicx}
\usepackage{color}
\usepackage{amsmath}
\usepackage{amssymb}         
\usepackage{amsfonts}
\usepackage{amsthm}

\theoremstyle{remark}
\newtheorem{remark}{Remark}[section]
\theoremstyle{definition}

\numberwithin{equation}{section}

\title[Schr\"odinger-Korteweg-de Vries system]{Dispersive blow-up and persistence properties for the Schr\"odinger-Korteweg-de Vries system}

\author[F. Linares]{Felipe Linares$^*$}
\thanks{$^{*}$ F.L. was partially supported by CNPq and FAPERJ/Brazil}
\address{Felipe Linares, IMPA Instituto Matem\'atica Pura e Aplicada, Estrada Dona Castorina 110, Rio de Janeiro, RJ 22460-320, Brazil}
\email{linares@impa.br}

\author[J. M. Palacios]{Jos\'e M. Palacios$^{**}$}
\thanks{$^{**}$ J.M.P.  was partially supported by Chilean research grants FONDECYT  1150202, Fondo Basal CMM-Chile and Millennium Nucleus Center for Analysis of PDE NC130017.}
\address{Jos\'e M. Palacios, Departamento de Ingeniería Matem\'atica DIM, Universidad de Chile}
\email{jpalacios@dim.uchile.cl}

\date{\today}

\newcommand{\be}{\begin{equation}}
\newcommand{\ee}{\end{equation}}
\newcommand{\bp}{\begin{proof}}
\newcommand{\ep}{\end{proof}}
\newcommand{\bel}{\begin{equation}\label}
\newcommand{\eeq}{\end{equation}}
\newcommand{\bea}{\begin{eqnarray}}
\newcommand{\eea}{\end{eqnarray}}
\newcommand{\bee}{\begin{eqnarray*}}
\newcommand{\eee}{\end{eqnarray*}}
\newcommand{\ben}{\begin{enumerate}}
\newcommand{\een}{\end{enumerate}}

\newcommand{\R}{\mathbb{R}}

\newcommand{\C}{\mathbb{C}}

\newcommand{\tres}{|\!|\!|}

\newcommand{\ji}{\langle}
\newcommand{\jd}{\rangle}

\newcommand{\p}{\partial}

\newtheorem{thm}{Theorem}[section]

\newtheorem{lem}[thm]{Lemma}

\newtheorem*{thmap}{{\bf Theorem A}}

\theoremstyle{remark}

\numberwithin{equation}{section}

\begin{document}


\begin{abstract}
We study the dispersive blow-up phenomena for  the Schr\"odinger-Korteweg-de 
Vries (S-KdV) system. Roughly, dispersive blow-up has being called to the development of point singularities due to the focussing of short or long waves.
In mathematical terms, we show that the existence of this kind of singularities is provided by the linear dispersive solution by proving that the Duhamel
term is smoother. It seems that this result is the first regarding systems of nonlinear dispersive equations. To obtain our results we use, in addition to smoothing properties, persistence properties for solutions of the IVP  in fractional weighted Sobolev spaces which we establish here.
\end{abstract}
\maketitle \markboth{ } {F. Linares and J. M. Palacios}

\medskip
\section{Introduction and main results}

\subsection{The model}

This paper is concerned with properties of solutions of the initial value problem (IVP) associated to the Schr\"odinger-Korteweg-de Vries (S-KdV) system, 
\begin{equation}\label{IVP}
\left\{\begin{array}{ll}
i \partial_tu+\partial_x^2u+\vert u\vert^2u=\alpha uv, \qquad t,x\in\mathbb{R},
\\ \partial_tv+\partial_x^3v+\tfrac{1}{2}\partial_x(v^2)=\gamma\partial_x(\vert u\vert^2),
\\ u(x,0)=u_0(x), \quad v(x,0)=v_0(x),
\end{array}\right.
\end{equation}
where $u=u(t,x)$ is a complex-valued function and $v(t,x)$ is a real-valued function. This system governs the interactions between shortwaves $u=u(t,x)$ and longwaves $v=v(t,x)$ and has been studied in several fields of physics and fluid dynamics (see \cite{FuOi,HoIkMiNi,KaKaSu,SaYa}).

\medskip 

The Schr\"odinger-Korteweg-de Vries system \eqref{IVP} has been shown not to be a completely integrable system (see \cite{BeBu}). Therefore the solvability of \eqref{IVP} is dependent upon the method of evolution equations. 

\medskip 

The IVP \eqref{IVP} has been extensively studied from the view point of local and global well-posedness.  Inspired in the results 
obtained for the famous
Korteweg-de Vries (KdV) (\cite{KPV2}) and the cubic Schr\"odinger equation (\cite{Ts})  several authors have studied the IVP \eqref{IVP}.
In general, a coupled system like \eqref{IVP} is more difficult to handle in the same spaces as in the space the single equation is solved. 
In the case of the system \eqref{IVP} this is due to the antisymmetric nature of the characteristics of each linear part. In \cite{BeOgPo2} Bekiranov, Ogawa and Ponce showed that the coupled system \eqref{IVP} is locally well-posed in $H^s(\mathbb{R})\times H^{s-\frac{1}{2}}(\mathbb{R})$ with $s\geq 0$. In \cite{CoLi} Corcho and Linares extended this result for weak initial data $(u_0,v_0)\in H^k(\mathbb{R})\times H^s(\mathbb{R})$ for various values of $k$ and $s$, where the lowest admissible values are $k=0$ and $s=-\tfrac{3}{4}+\delta$ with $0<\delta\leq \tfrac{1}{4}$. The end-point $(k,s)=(0,-\tfrac{3}{4})$ was treated in \cite{ZiYu} by Z. Guo and Y. Wang . We observe that no local/global well-posedness results in weighted Sobolev spaces have been registered in the literature as far as we know.

\medskip 

The aim of this work is to study the dispersive blow-up for solutions of the S-KdV system.  In \cite{BoSa}  Bona and Saut started the mathematical 
analysis of the dispersive blow-up for solutions of the generalized KdV equation. More precisely, they proved the following

\begin{thmap}[\cite{BoSa}] Let $T>0$ be given and let $\{(x_n,t_n)\}_{n=1}^{\infty}$ be  a sequence of points in $\R\times(0,T)$ without finite limit points and such that $\{t_n\}_{n=1}^{\infty}$ is bounded below by a positive constant. Let either $s=0$ and $k=1$ or $s\ge 2$ and $k\ge 1$ an arbitrary integer. Then there exists $\psi\in H^s(\R)\cap C^{\infty}(\R)$ such that the solution of the IVP
\begin{equation*}
\begin{cases}
u_t+u^ku_x+u_{xxx}=0,\\
u(x,0)=\psi(x)
\end{cases}
\end{equation*}
satisfies
\begin{enumerate}
\item $u$ lies in $L^{\infty}([0,T]: H^s(\R))\cap L^2([0,T] :H^{s+1}_{\rm loc}(R))$, or in $C([0,T]: H^s(\R))\cap L^2([0,T] :H^{s+1}_{\rm loc}(R))$, if $s\ge 2$.
\item $\p_x^su$ is continuous on $\R\times (0,T)\backslash U_{n=1}^{\infty}\{(x_n,t_n)\}$, and
\item $\lim_{(x,t)\to (x_n,t_n)} \p_x^s u(x,t)=+\infty$ for $n=1,2,\dots$.
\end{enumerate}
\end{thmap}

The main idea behind the proof is to show that the Duhamel term associated to the solution of the IVP is smoother than the linear term of the
solution.  In \cite{LiSc} Linares and Scialom proved for $k\ge 2$ by means of the smoothing effects established for the linear KdV equation without using 
weighted Sobolev spaces. Recently, Linares, Ponce and Smith \cite{LiPoSm} using fractional weighted spaces improved the previous result
in the case $k=1$, i.e., for the KdV equation. 

The analogous phenomena also appears in other linear dispersive equations, such as the linear Schr\"odinger equation and the free surface water waves system linearized around the rest state \cite{BoSa2}. In \cite{BoSa2} Bona and Saut constructed initial data with point singularities for solutions of
the linear Schr\"odinger equation. Bona, Ponce, Saut and Sparber \cite{BoPoSaSp}  established the dispersive blow-up for the semilinear Schr\"odinger equation in dimension $n$ and other Schr\"odinger type equations. The main tools employed to show these results were the intrinsic smoothing effects
of these dispersive equations. We shall remark that the only $n$-dimensional result regarding dispersive blow-up is this one just above refereed for the nonlinear Schr\"odinger
equation.

\subsection{Main results.}   Inspired in the dispersive blow-up results for the KdV and Schr\"odinger equations it was natural to ask what was the situation for solutions for the Schr\"odinger-Korteweg-de Vries system concerning this property. In our study we got the following answer.

\begin{thm}\label{MT2}
There exist initial data \[
u_0\in C^\infty(\mathbb{R})\cap H^{2^-}(\mathbb{R}), \quad  v_0\in C^\infty(\mathbb{R})\cap H^{3/2^-}(\mathbb{R}),
\]
for which the corresponding solution $(u,v)(\cdot,\cdot)$ of the IVP \eqref{IVP} provided by Theorem \ref{MT1} (below):\[
u\in C([0,T]:H^{2^-}(\mathbb{R})), \quad v\in C([0,T]:H^{3/2^-}(\mathbb{R})),
\]
satisfies that there exists $t^*\in[0,T]$ such that \[
u(\cdot,t^*)\notin C^{1,\frac{1}{2}^+}(\mathbb{R}),\quad  v(\cdot,t^*)\notin C^1(\mathbb{R}).
\]
\end{thm}

To prove this result, we construct first  initial data  lending some ideas in \cite{BoSa2} and  \cite{LiPoSm}. To treat the nonlinear problem is not straight forward, as we shall see, in our case the NLS-KdV system presents several new difficulties because its coupling terms. In addition to the
smoothing effects, the new key ingredient in our arguments is the  persistence property of solutions of the IVP \eqref{IVP} on weighted spaces, which allow us to close some nonlinear estimates for the solution.

\subsubsection{Persistence properties} Due to the presence of the KdV structure in the system we need to use weighted spaces in order to
show that the Duhamel term is smoother than the linear part of the equation. 

As we commented above even in the usual Sobolev spaces 
the coupling of the Schr\"odinger equation and KdV equation introduces some difficulties because of the structure of the \lq\lq symbols" of the linear equations.
To complete our analysis in the dispersive blow-up result we need the following result which includes local well-posedness of the IVP \eqref{IVP}
in fractional Sobolev spaces and a persistence property of these solutions in weighted spaces. More precisely,

\begin{thm}\label{MT1}
Let $s,r_1,r_2$ be positive numbers such that $s>3/4$, $s+1/2\ge r_1$ and $s\ge 2r_2$ and consider initial data\[
(u_0,v_0)\in H^{s+\frac{1}{2}}(\mathbb{R})\cap L^2(\vert x \vert ^{2r_1}dx)\times H^{s}(\mathbb{R})\cap L^2(\vert x\vert^{2r_2}dx).
\]
Then there exist $T=T(\Vert u_0\Vert_{s+\frac{1}{2}}+\Vert v_0\Vert_s)>0$ and a unique solution $(u(t),v(t))$ of the IVP \eqref{IVP} satisfying 
\begin{equation}\label{mteq1}
u\in C([0,T]; H^{s+\frac{1}{2}}(\mathbb{R})\cap L^2(\vert x\vert ^{2r_1}dx)), \quad v\in C([0,T]; H^s(\mathbb{R})\cap L^2(\vert x\vert^{2r_2}dx)), 
\end{equation}
\begin{equation}\label{mteq2}
\Vert D^s_x\partial_xu\Vert_{L^\infty_xL^2_T}+\Vert D^{s-\frac{1}{2}}_x\partial_xv\Vert_{L^\infty_xL^2_T}+\Vert D^s_x\partial_x v\Vert_{L^\infty_xL^2_T}<\infty,
\end{equation}
\begin{equation}\label{mteq3}
\Vert u\Vert_{L^2_xL^\infty_T}+\Vert v\Vert_{L^2_x L^\infty_T}<\infty, 
\end{equation}
\begin{equation}\label{mteq4}
\Vert \partial_x u\Vert_{L^4_TL^\infty_x}+\Vert \partial_xv \Vert_{L^4_TL^\infty_x}<\infty,
\end{equation}
with the local existence time satisfying: \[
T=T(\Vert u_0\Vert_{s+\frac{1}{2}}+\Vert v_0\Vert_s)\to+\infty \quad \text{as} \quad \Vert u_0\Vert_{s+\frac{1}{2}}+\Vert v_0\Vert_s \to 0.
\]
Moreover, given $T'\in(0,T)$, the map data solution $(u_0,v_0)\mapsto (u,v)(t)$ from $H^{s+\frac{1}{2}}(\mathbb{R})\cap L^2(\vert x\vert^{2r_1}dx)\times H^s(\mathbb{R})\cap L^2(\vert x\vert^{2r_2})$ to the class defined by \eqref{mteq1}-\eqref{mteq2} is Lipschitz continuous.
\end{thm}

The proof of Theorem \ref{MT1} uses the contraction mapping principle which is combined with smoothing properties of solutions of the  associated
linear problems for the Schr\"odinger and KdV equations.  The key ingredient in our analysis to prove the persistence property  is 
 a new pointwise formula that allows to commute  the fractional weights $\vert x\vert^s$  with the Schr\"odinger group $e^{it\Delta}$  and the
 Airy group $e^{-t\partial_x^3}$. This pointwise formula was deduced by Fonseca, Linares and Ponce in \cite{FoLiPo}.
 
 \begin{remark} The result in Theorem \ref{MT1} is not available in the literature and as the case for the dispersive blow-up it seems the
 first one for systems.
 \end{remark}

Next we introduce some notation we will utilize along this work.

\subsection{Notation}

Let $1\leq p,q\le \infty$ and $f:\mathbb{R}\times[0,T]\to\mathbb{R}$. We define the norm
$$
\Vert f\Vert_{L^p_xL^q_T}:=\left(\int_\mathbb{R}\bigg(\int_0^T \vert f\vert^{q}dt\bigg)^{p/q}dx\right)^{1/p},
$$
with the usual modifications when $p=\infty$ or $q=\infty$. Similarly for $\Vert f\Vert_{L^p_TL^q_x}$. 

We will denote the  homogeneous derivatives of order $s>0$ by
\begin{align*}
D^s_xf(x):=\mathcal{F}^{-1}\big(\vert \xi\vert^s\hat f(\xi)\big)(x).
\end{align*}
where $\widehat{f}$ denotes the Fourier transform of $f$ and $\mathcal{F}^{-1}$ the inverse Fourier transform.
As usual for $s\in \R$ we shall denote by $H^s(\mathbb{R})$ the standard $L^2$-based Sobolev space:
$$
H^s(\mathbb{R}):=\{f\in \mathcal{S}' (\mathbb{R}): \ \Vert f\Vert_{s,2}=\Vert J^sf\Vert_2<\infty\},
$$
where
$$
\Vert J^sf\Vert_2:=\left(\int_{-\infty}^\infty (1+\xi^2)^s\vert \widehat{f}(\xi)\vert^2d\xi\right)^{1/2}.
$$

Finally, in the reminder of this work we will adopt the following notation, for $s\geq 0$ \[
H^{s^+}(\mathbb{R})=\bigcup_{s'>s}H^{s'}(\mathbb{R}), \qquad H^{s^-}(\mathbb{R})=\bigcup_{0\leq s'<s}H^{s'}(\mathbb{R}).
\]

\subsection{Organization of this paper}

This paper is organized as follows. In Section 2 we state a series of results needed in our analysis. In Section 3 we construct the initial data which develop dispersive blow-up. In Section 4 we establish Theorem \ref{MT1}. Finally, in Section 5 we show  our main result Theorem \ref{MT2}.

\bigskip

\section{Preliminaries}

\medskip

\subsection{Smoothing properties}
In this subsection some technical results on the smoothing properties of the free Schr\"odinger group $S(t):=e^{ it\Delta}$ and the KdV group $V(t):=e^{-t\partial_x^3}$ are reviewed. They will find use in Section $3$, $4$ and $5$.

\medskip 

Next lemma provides the smoothing effects of Kato type for solutions of the linear KdV equation.

\begin{lem}[\cite{KPV}]

\begin{equation}\label{kato-s-kdv}
\sup_{x}\| \partial_x V(t)v_0\|_{L^2_t}\le c\|v_0\|_{L^2}
\end{equation}
and
\begin{equation}\label{dual-kato-s-kdv}
\|\partial_x \int_0^t V(t-t')F(\cdot, t')\,dt'\|_{L^2_x}\le c\|F\|_{L^1_xL^2_T}.
\end{equation}
\end{lem}

Next lemma gives us the smoothing effects of Kato type for solutions of the linear Schr\"odinger equation in dimension $n=1$. 
\begin{lem}[\cite{KPV}]\label{smoothing_kato}

\begin{equation}\label{kato-s-schr}
\sup_{x}\| D^{1/2}_x e^{it\Delta}u_0\|_{L^2_t}\le c\|u_0\|_{L^2},
\end{equation}
\begin{equation}\label{dual-kato-s-schr1}
\|D^{1/2}_x \int_0^t e^{i(t-t')\Delta} F(\cdot, t')\,dt'\|_{L^2_x}\le c\|F\|_{L^1_xL^2_T},
\end{equation}
and
\begin{equation}\label{dual-kato-s-schr2}
\sup_{x}\|\partial_x \int_0^t e^{i(t-t')\Delta}F(\cdot, t')\,dt'\|_{L^2_T}\le c\|F\|_{L^1_xL^2_T}.
\end{equation}
\end{lem}
Next we present Strichartz estimates for both groups $\{e^{it\Delta}\}$ and $V(t)$.
\begin{lem}
Let $2\leq p,q\leq\infty$ such that $\tfrac{2}{q}=\tfrac{1}{2}-\tfrac{1}{p}$. Then the following holds \begin{align}\label{str-sch}
\Vert e^{it\Delta}f\Vert_{L^q_tL^p_x}\leq c\Vert f\Vert_{L^2_x}.
\end{align}
\end{lem}

\begin{lem} 
For $(\alpha,\theta)=[0,1/2]\times[0,1]$ it holds that \begin{align}\label{str-kdv}
\Vert D^{\alpha\theta/2}_xV(t)f\Vert_{L^q_tL^p_x}\leq c\Vert f\Vert_{L^2_x}
\end{align}
where $(q,p)=(6/\theta(\alpha+1),2/(1-\theta)))$.
\end{lem}
For a proof of these estimates see for instance \cite{LiPo}. Finally, we complete the set of estimates introducing the next maximal function estimate for the linear solutions. 

\begin{lem}\label{maximal_est}
For $s>1/2$ and $\rho_1>1/4$, it holds that 
\begin{align}\label{sch-mfn}
\Vert e^{it\Delta}f\Vert_{L^2_xL^\infty_T}\leq c(1+T)^{\rho_1}\Vert f\Vert_{s,2}.
\end{align}
For $s\ge 1/4$, it holds
\begin{align}\label{sch-mfn-4}
\Vert e^{it\Delta}f\Vert_{L^4_xL^\infty_T}\leq c\,\Vert f\Vert_{s,2}.
\end{align}
For $s>3/4$ and $\rho_2>3/4$ it holds that \begin{align}\label{kdv-mfn}
\Vert V(t)f\Vert_{L^2_xL^\infty_T}\leq c(1+T)^{\rho_2}\Vert f\Vert_{s,2}.
\end{align}
It holds
\begin{align}\label{mkdv-mfn}
\Vert V(t)f\Vert_{L^4_xL^\infty_T}\leq c \Vert D^{1/4}_xf\Vert_{L^2_x}.
\end{align}
\end{lem}

\begin{proof}
See \cite{RoVi,Ve} for a proof of \eqref{sch-mfn} and \eqref{sch-mfn-4}. For a proof of \eqref{kdv-mfn} see \cite{Ve2}. For a proof of  
\eqref{mkdv-mfn} see \cite{KPV}.
\end{proof}

To end this subsection we have the following interpolated estimates
\begin{lem}\label{interpolation-kdv}
\begin{equation}\label{inter-1}
\|V(t)f\|_{L^5_xL^{10}_t}\le c\|f\|_{L^2_x}
\end{equation}
and
\begin{equation}\label{inter-2}
\|D^{1/2}_xV(t)f\|_{L^{20/3}_xL^5_T}\le c\|D^{1/4}_xf\|_{L^2_x}.
\end{equation}
\end{lem}

\begin{proof} The estimates \eqref{inter-1} and \eqref{inter-2} follow by interpolating \eqref{kato-s-kdv} and
\eqref{mkdv-mfn}. See \cite{KPV}.
\end{proof}

\medskip

\subsection{Weighted Estimates}

Since we are going to deal with weighted spaces, the next interpolation estimate will be very useful.
\begin{lem}[\cite{NP}]\label{nahas-ponce} Let $a,b>0$. Assume that $J^af=(1-\Delta)^{a/2}f\in L^2(\mathbb{R}^n)$ and $\langle x\rangle^bf=(1+\vert x\vert^2 ) ^{b/2}f\in L^2(\mathbb{R}^n)$. Then, for any $\theta\in(0,1)$ 
\begin{equation}\label{nahas-ponce-ineq}
\Vert \langle x\rangle^{\theta b}J^{(1-\theta)a}f\Vert_2\leq c\Vert \langle x\rangle^bf\Vert_2^\theta\Vert J^af\Vert_2^{1-\theta}.
\end{equation}
and
\begin{equation}\label{nahas-ponce-ineq2}
\|\langle x \rangle^{(1-\theta)b}D_x^{\theta a}f\|_{L^2_x}\le c\, \|\langle x \rangle^bf\|_{L^2_x}^{1-\theta}\|D^af\|_{L^2_x}^\theta.
\end{equation}
\end{lem}

\medskip

Now, we recall a very useful formula derived in \cite{FoLiPo} for the Airy group: for $\beta\in (0,1)$ and $t\in\mathbb{R}$, the following formula holds: \begin{align}\label{kdv_persistencia}
\vert x\vert^\beta V(t)f= V(t)\big(\vert x\vert^\beta f\big)+V(t)\{\Phi_{t,\beta}(\hat f)\},
\end{align}
with 
\begin{align}\label{pers_kdv2}
\Vert \Phi_{t,\beta}(\hat f)\Vert_2\leq c(1+\vert t\vert)\Vert f\Vert_{2\beta,2}.
\end{align}
Using the same arguments as in \cite{FoLiPo}, we can also deduce a formula for solutions of the linear Schr\"odinger equation.
\begin{lem}
\
\begin{enumerate}
\item Let $\beta\in(0,1)$ and $t\in\R$. Then the following pointwise formula holds
\begin{equation}\label{lwp-12}
|x|^\beta e^{it\Delta} u_0 (x)= e^{it\Delta} ( |x|^\beta u_0) + e^{it\Delta} \big\{\Lambda_{t,\beta}(\widehat{u}_0)(\xi)\big\}^{\vee},
\end{equation}
\noindent where \medskip
\begin{equation}\label{lwp-13}
\|\big\{\Lambda_{t,\beta}(\widehat{u}_0)(\xi)\big\}^{\vee}\|_{L^2_x}\le c(1+|t|)\,(\|u_0\|_{L^2_x}+\|D^\beta_xu_0\|_{L^2_x}\big).
\end{equation}

\item Let $\beta\in(0,1)$, $t\in\R$ and $(p,q)$ such that $\tfrac{2}{q}=\tfrac{1}{2}-\tfrac{1}{p}$. Then,  
\begin{align}\label{weighted-sch}
\Vert \vert x\vert^\beta e^{it\Delta}u_0\Vert_{L^q_tL^p_x}\leq c\Vert \vert x\vert^\beta u_0\Vert_{L^2_x}+c(1+\vert t\vert)\big(\Vert u_0\Vert_{L^2_x}+\Vert D^\beta_xu_0\Vert_{L^2_x}\big).
\end{align}
\end{enumerate}
\end{lem}

\begin{proof}
The first part of the lemma follows the same arguments used to prove \eqref{kdv_persistencia} in  \cite{FoLiPo}. For the second part we use Strichartz estimate \eqref{str-sch} and the pointwise formula \eqref{lwp-12} to obtain 
\begin{align*}
\Vert \vert x\vert^\beta e^{it\Delta}u_0\Vert_{L^q_tL^p_x}&\leq \Vert e^{it\Delta}(\vert x\vert^\beta u_0)\Vert_{L^q_tL^p_x}+\Vert e^{it\Delta}\{\Lambda_{t,\beta}(\hat u_0)(\xi)\}^\vee\Vert_{L^q_tL^p_x}
\\ & \leq c\Vert \vert x\vert^\beta u_0\Vert_{L^2_x}+c(1+\vert t\vert)\big(\Vert u_0\Vert_{L^2_x}+\Vert D^\beta_xu_0\Vert_{L^2_x}\big),
\end{align*}
which concludes the proof.
\end{proof}

These estimates \eqref{pers_kdv2}-\eqref{lwp-13} will be crucial in the proof of Theorem \ref{MT1}.

\medskip

\subsection{Leibnitz rule}

To end this section we present the fractional Leibnitz rule which will be employed to deal with nonlinear terms.
\begin{thm}[\cite{KPV}]\label{leibnitz}
\ 
\begin{enumerate}
\item For $s>0$ and $1<p<\infty$, it holds \begin{align}\label{commutator1}
\Vert D^s_x(fg)-fD^s_xg-gD^s_xf\Vert_p\leq c\Vert f\Vert_{\infty}\Vert D^s_xg\Vert_p.
\end{align}
\item Let $\alpha\in(0,1)$, $\alpha_1,\alpha_2\in[0,\alpha]	$ with $\alpha=\alpha_1+\alpha_2$. Let $p_1,p_2,q_1,q_2\in(1,\infty)$ with $1=\tfrac{1}{p_1}+\tfrac{1}{p_2}$ and $\tfrac{1}{2}=\tfrac{1}{q_1}+\tfrac{1}{q_2}$. Then, 
\begin{align}\label{LR-L1L2}
\Vert D^\alpha_x(fg)-fD^\alpha_xg-gD^\alpha_xf\Vert_{L^1_xL^2_T}\leq c\Vert D^{\alpha_1}_xf\Vert_{L^{p_1}_xL^{q_1}_T}\Vert D^{\alpha_2}_xg\Vert_{L^{p_2}_xL^{q_2}_T}.
\end{align}
\end{enumerate}
\end{thm}

\bigskip

\section{Construction of the initial data}\label{section_ID}
 
In this section, attention is turned to understand dispersive blow-up for each linear equation.
\medskip

Let us divide the analysis in two cases, the linear case of the Schr\"odinger equation and the linear case of the KdV equation: 

\medskip

\subsection{Linear case: Schr\"odinger equation.} We will follow the argument employed in \cite{BoPoSaSp} and \cite{BoSa2} with some modifications. 

\bigskip

Consider the IVP associated to the linear Schr\"odinger equation: \begin{align}\label{lsch}
\begin{cases} 
i\partial_tu+\partial_x^2u=0,
\\ u(x,0)=u_0(x).
\end{cases}
\end{align}
Now, recall that for any $u_0\in L^2(\mathbb{R})$, the unique solution $u$ of \eqref{lsch} has the representation: \[
u(x,t)=\dfrac{1}{(4i\pi t)^{1/2}}\int_{\mathbb{R}}e^{\tfrac{i\vert x-y\vert^2}{4t}}u_0(y)\,dy,
\]
where the integral is taken in the improper Riemann sense. 

Let $u_0$ be defined as
\begin{align}\label{initial_data_schr}
u_0(x):=\dfrac{e^{- i\alpha(x-x_0)^2}}{(1+x^2)^{5/4}}, \qquad \alpha >0.
\end{align}
It is not difficult to show that  $u_0\in H^s(\mathbb{R})$ for any $s\in[0,2)$ but  $u_0\notin H^2(\mathbb{R})$. Moreover, $u_0$ satisfies 
\[
\langle x\rangle ^{2^-}u_0(x)\in L^2(\mathbb{R}). 
\] 
(See \cite{BoPoSaSp, BoSa2} for the details).

In the next lemma we present a precise statement of the dispersive blow-up for the linear Schr\"odinger equation.

\begin{lem}\label{dbu_lsch} 
Let $\alpha\in\mathbb{R}$ and $x_0\in\mathbb{R}$ fixed and let $\varepsilon\in(0,\tfrac{1}{2})$. Consider the point $(x^*,t^*)=(x_0,\tfrac{1}{4\alpha})$ and the initial data \[
u_0(x):=\dfrac{e^{-i\alpha (x-x_0)^2}}{(1+x^2)^\frac{5}{4}}
.
\]
Then, the initial data satisfies: \[ 
u_0\in C^\infty(\mathbb{R})\cap H^{2^-}(\mathbb{R})\cap L^\infty(\mathbb{R}),
\] and the associated global in-time solution $u\in C(\mathbb{R}; H^{2^-}(\mathbb{R}))$ of \eqref{lsch} has the following properties: \begin{enumerate}
\item For any time $t\in\mathbb{R}$ with $t\neq t^*$, the solution $u(\cdot,t)\in C^{1,\frac{1}{2}+\varepsilon}(\mathbb{R})$.
\item At time $t^*$, the solution $u(\cdot,t^*)\in C^{1,\frac{1}{2}+\varepsilon}(\mathbb{R}\setminus\{x^*\})$.
\item At time $t^*$, the solution $u(\cdot,t^*)\notin C^{1,\frac{1}{2}+\varepsilon}(\mathbb{R})$.
\end{enumerate}
\end{lem}
\begin{proof}
See \cite{BoPoSaSp} or \cite{BoSa2} for a detailed proof.
\end{proof}

This concludes the case of the free Schr\"odinger equation. Now the attention is turned to construct the initial data for the linear Korteweg-de Vries equation.

\bigskip

\subsection{Linear case: Korteweg-de Vries equation.} For this case we recall the data constructed in the proof of dispersive blow-up for the KdV equation given in \cite{LiPoSm} (see Section 3).

\medskip

Consider the linear IVP associated to the linear Korteweg-de Vries equation: \begin{align}\label{lkdv}
\begin{cases}
\partial_t v+\partial_x^3v=0, \qquad x\in\mathbb{R},\,t>0,
\\ v(x,0)=v_0(x),
\end{cases}
\end{align}
whose solution is given by \[
v(x,t)=V(t)v_0(x)=e^{-t\partial_x^3}v_0=S_t*v_0(x),
\]
where, \[
S_t(x):=\dfrac{1}{3\sqrt{3}t}A_i\left(\dfrac{x}{3\sqrt{3}t}\right),
\]
and $A_i(\cdot)$ denotes the Airy function. The following lemma give us the detailed statement for the dispersive blow-up for the initial-value problem associated to the linear KdV equation \eqref{lkdv}. 
\begin{lem}[\cite{LiPoSm}]\label{dbu_lkdv} Let $\alpha\in\R$ fixed and consider the initial data \[
v_0(x):= \sum_{j=1}^\infty \lambda_jV(-\alpha j)\phi(x), \quad \lambda_j>0,
\]
where $\lambda_j=c e^{-\alpha^2j^2}$ with $c>0$ small enough and $\phi(x):=e^{-2\vert x\vert}$. Then, \[ 
v_0\in C^\infty(\mathbb{R})\cap H^{3/2^-}(\mathbb{R})\cap L^2(\langle x\rangle ^{3/2^-}dx)\cap L^\infty(\mathbb{R}),
\] and the associated global in-time solution $v\in C(\mathbb{R}; H^{3/2^-}(\mathbb{R}))$ of \eqref{lkdv} has the following properties: \begin{enumerate}
\item For any $t>0$ with $t\notin \alpha\mathbb{Z}$, we have $v(\cdot,t)\in C^1(\mathbb{R})$.
\item For any $t\in\alpha\mathbb{N}$ we have $v(\cdot,t)\notin C^1(\mathbb{R})$.
\end{enumerate}
\end{lem}

\begin{proof} For a detailed proof of this statement see \cite{LiPoSm}, section 3.
\end{proof}

\bigskip

\section{Proof of Theorem \ref{MT1}}

In this section we show the persistence property in weighted spaces of solutions of the IVP \eqref{IVP}.

\begin{proof}[Proof of Theorem \ref{MT1}]

The idea of the proof is to apply the contraction principle to the system of integral equations equivalent to \eqref{IVP}, that is,
\begin{equation}\label{I-S-KdV}
\begin{split}
\Phi(u) &= S(t)u_0 +\int_0^t S(t-t') (uv)(t')\,dt' +\int_0^t S(t-t') |u|^2u(t')\,dt',\\
\Psi(v) &= V(t) v_0 -\int_0^t V(t-t') v\p_xv(t')\,dt' + \int_0^t V(t-t') \p_x(|u|^2)(t')\, dt',
\end{split}
\end{equation}
where $\{S(t)\}$ and $\{V(t)\}$ are the unitary groups associated to the linear Schr\"odinger and the Airy equation respectively.
We will give a sketch of the proof.

For $u:\R\times[0,T]\to \C$ and $v:\R\times[0,T]\to \R$ with $T$ fixed and $s>3/4$, define
\begin{equation}\label{spaces-1}
\mu_1^T(u):= \|u\|_{L^{\infty}H^{s+1/2}}+\|D_x^sD^{1/2}_x u\|_{L^{\infty}_xL^2_T}+\|u\|_{L^2_xL^{\infty}_T}+\|\p_x u\|_{L^4_TL^{\infty}_x},
\end{equation}
\begin{equation}\label{spaces-2}
\begin{split}
\mu_2^T(v):= &\|v\|_{L^{\infty}H^{s}}+\|D_x^s\p_xv\|_{L^{\infty}_xL^2_T}
+\|D_x^{s-1/2}\p_xv\|_{L^{\infty}_xL^2_T}\\ &+\|v\|_{L^2_xL^{\infty}_T}
+\|\p_x v\|_{L^4_TL^{\infty}_x}.
\end{split}
\end{equation}

By using the definition, group properties, Minkowski's inequality, and Sobolev spaces properties we have
\begin{equation}\label{lwp-1}
\begin{split}
&\|D^{s+\frac12}_x \Phi(u)\|_{L^2}\\ &\le c\|D^{s+\frac12}_x u_0\|_{L^2}+\int_0^T \|D^{s+\frac12}_x (uv)\|_{L^2}\,dt'+\int_0^T \|D^{s+\frac12}_x (|u|^2u)\|_{L^2}\,dt'\\
&\le c\,\|u_0\|_{s+1/2} +\int_0^T \|D^{s-\frac12}_x \p_x(uv)\|_{L^2}\,dt'+ c\,T\,\underset{[0,T]}{\sup}\|u(t)\|_{s+1/2}^3.
\end{split}
\end{equation}

To complete the estimate we use the commutator estimate \eqref{commutator1}, Sobolev spaces properties, the Cauchy-Schwarz inequality, and H\"older's inequality
in time to led to
\begin{equation}\label{lwp-2}
\begin{split}
&\int_0^T \|D^{s-\frac12}_x \p_x(uv)\|_{L^2}\,dt \\
&\le c\,\int_0^T\|D^{s-1/2}_x(u\p_xv)\|_{L^2}\,dt+c\,\int_0^T\|D^{s-1/2}_x(\p_xu v)\|_{L^2}\,dt
\\ &  \le c\,\int_0^T \|\p_x v\|_{L^{\infty}_x}\|D^{s-1/2}_xu\|_{L^2}\,dt +c\,\int_0^T\| u D^{s-1/2}_x\p_x v\|_{L^2}
\\ &  \quad +c\,\int_0^T \|\p_x u\|_{L^{\infty}_x}\|D^{s-1/2}_xv\|_{L^2}\,dt +c\,\int_0^T\| v D^{s-1/2}_x\p_x u\|_{L^2}
\\ &  \le cT^{3/4}\,\|\p_x v\|_{L^4_TL^{\infty}_x}\underset{[0,T]}{\sup} \|u(t)\|_{s+1/2}+cT^{1/2} \|u\|_{L^2_xL^{\infty}_T}\|D^{s-1/2}_x\p_xv\|_{L^{\infty}_xL^2_T}
\\ &    \quad  + cT^{3/4}\,\|\p_x u\|_{L^4_TL^{\infty}_x}\underset{[0,T]}{\sup} \|v(t)\|_{s+1/2}+cT^{1/2} \|v\|_{L^2_xL^{\infty}_T}\|D^{s-1/2}_x\p_xu\|_{L^{\infty}_xL^2_T}.
\end{split}
\end{equation}

Combining \eqref{lwp-1} and \eqref{lwp-2} it follows that
\begin{equation}\label{lwp-2b}
\|D^{s+\frac12}_x \Phi(u)\|_{L^2}\le c\|u_0\|_{s+1/2}+ cT^{1/2}\big((1+T^{1/4})\mu_1^T(u)\mu_2^T(v)+ T^{1/2}(\mu_1^T(u))^3\big).
\end{equation}

Next we estimate the $H^s$-norm of $\Psi(v)$. It is enough to estimate $\|D^s_x\Psi(v)\|_{L^2}$.  To do so, we use group properties and Minkowskii's inequality
to obtain
\begin{equation}\label{lwp-3}
\begin{split}
\|D^{s}_x \Psi(v)\|_{L^2} &\le c\|D^{s}_x v_0\|_{L^2}+\int_0^T \|D^s_x (v\p_xv)\|_{L^2}\,dt'+\int_0^T \|D^{s}_x \p_x(u\bar{u})\|_{L^2}\,dt'\\
&\le c\,\|v_0\|_{s} +\int_0^T \|D^{s}_x (v\p_xv)\|_{L^2}\,dt'+ \int_0^T \|D^{s}_x (\bar{u}\p_x u)\|_{L^2}\,dt'\\
&\;\;\;+\int_0^T \|D^{s}_x (u\p_x\bar{u})\|_{L^2}\,dt'.
\end{split}
\end{equation}

The commutator estimates \eqref{commutator1} and Holder's inequality yield 
\begin{equation}\label{lwp-4}
\begin{split}
\int_0^T \|D^s_x (v\p_xv)\|_{L^2}\,dt &\le cT^{3/4}\,\|\p_x v\|_{L^4_TL^{\infty}_x}\underset{[0,T]}{\sup} \|v(t)\|_{s}\\
&\quad+cT^{1/2} \|v\|_{L^2_xL^{\infty}_T}\|D^s_x\p_xv\|_{L^{\infty}_xL^2_T}.
\end{split}
\end{equation}
Similarly, we get
\begin{equation}\label{lwp-5}
\begin{split}
& \int_0^T  \|D^{s}_x (\bar{u}\p_x u)\|_{L^2}\,dt'+\int_0^T \|D^{s}_x (u\p_x\bar{u})\|_{L^2}\,dt'\\
&  \qquad  \qquad \le cT^{3/4}\,\|\p_x u\|_{L^4_TL^{\infty}_x}\underset{[0,T]}{\sup} \|u(t)\|_{s}+cT^{1/2} \|u\|_{L^2_xL^{\infty}_T}
 \|D^s_x\p_xu\|_{L^{\infty}_xL^2_T}.
\end{split}
\end{equation}

Using the definition \eqref{spaces-2} and the inequalities \eqref{lwp-3}, \eqref{lwp-4} and \eqref{lwp-5} we deduce that
\begin{equation}\label{lwp-5b}
\|D^{s}_x \Psi(v)\|_{L^2} \le c\|v_0\|_s+ c T^{1/2} (1+T^{1/4})\big((\mu_2^T(v))^2 +(\mu_1^T(u))^2\big).
\end{equation}

On the other hand, use of Kato's smoothing effect \eqref{kato-s-schr} and the analysis in \eqref{lwp-1} and \eqref{lwp-2} yield
\begin{equation}\label{lwp-6}
\begin{split}
&\|D^s_x\p_x \Phi(u)\|_{L^{\infty}_xL^2_T} \\
&\le c\|u_0\|_{s+1/2} + cT^{1/2}\big((1+T^{1/4})\mu_1^T(u)\mu_2^T(v)+ T^{1/2}(\mu_1^T(u))^3\big).
\end{split}
\end{equation} 

Same argument as above, now applying Kato's smoothing effect \eqref{kato-s-kdv} and the arguments in \eqref{lwp-3}, \eqref{lwp-4} and
\eqref{lwp-5} lead to
\begin{equation}\label{lwp-7}
\begin{split}
\|D^s_x\p_x \Psi(v)\|_{L^{\infty}_xL^2_T} &\le c\|v_0\|_{s} +  c T^{1/2} (1+T^{1/4})\big((\mu_2^T(v))^2 +(\mu_1^T(u))^2\big).
\end{split}
\end{equation} 

From the maximal norm estimates \eqref{sch-mfn} and \eqref{kdv-mfn} combined with the arguments in \eqref{lwp-1}, \eqref{lwp-2}
and \eqref{lwp-3}, \eqref{lwp-4}, \eqref{lwp-5} it follows that 
\begin{equation}\label{lwp-8}
\begin{split}
\|\Phi (u)\|_{L^2_xL^{\infty}_T} &+\|\Psi (v)\|_{L^2_xL^{\infty}_T}\\
&\le c\, (1+T)^{\rho_1}\|u_0\|_{s+1/2} +c\, (1+T)^{\rho_2}\|v_0\|_{s} \\
&\quad+ cT^{1/2}\,(1+T)^{\rho_1}\,\big((1+T^{1/4})\mu_1^T(u)\mu_2^T(v)+ T^{1/2}(\mu_1^T(u))^3\big)\\
&\quad + c T^{1/2}(1+T)^{\rho_2} (1+T^{1/4})\big((\mu_2^T(v))^2 +(\mu_1^T(u))^2\big).
\end{split}
\end{equation}

The Strichartz estimates \eqref{str-sch} and \eqref{str-kdv} together with the analysis in \eqref{lwp-1}, \eqref{lwp-2}
and \eqref{lwp-3}, \eqref{lwp-4}, \eqref{lwp-5} lead to 
\begin{equation}\label{lwp-10}
\begin{split}
\|\p_x\Phi (u)\|_{L^4_TL^{\infty}_x} &+\|\p_x\Psi (v)\|_{L^4_TL^{\infty}_x}
 \le c\, (\|u_0\|_{s+1/2} + \|v_0\|_{s}) \\
&\hskip10pt + cT^{1/2}\big((1+T^{1/4})\mu_1^T(u)\mu_2^T(v)+ T^{1/2}(\mu_1^T(u))^3\big)\\
&\hskip10pt +  cT^{1/2}(1+T^{1/4})\big((\mu_2^T(v))^2 +(\mu_1^T(u))^2\big).
\end{split}
\end{equation}

Combining \eqref{lwp-2}, \eqref{lwp-5b}, \eqref{lwp-6}, \eqref{lwp-7}, \eqref{lwp-8}, \eqref{lwp-10}, and the definitions \eqref{spaces-1} and 
 \eqref{spaces-2}, we have
\begin{equation}\label{contraction}
\begin{split}
\tres(\Phi(u),\Psi(v))\tres&=\mu_1^T(\Phi(u))+\mu_2^T(\Psi(v))\\
&\le c\,(\|u_0\|_{s+1/2}+\|v_0\|_s) \\
&\hskip10pt + cT^{1/2}\big(\mu_1^T(u)\mu_2^T(v)+ (\mu_1^T(u))^3\big)\\
&\hskip10pt +  cT^{1/2}\big((\mu_2^T(v))^2 +(\mu_1^T(u))^2\big)\\
&\hskip10pt +  cT^{1/2}\big (\tres (u,v)\tres^2+\tres (u,v)\tres^3\big).
\end{split}
\end{equation}
for $0<T<1$. Choosing  $a\le 2c\,(\|u_0\|_{s+1/2}+\|v_0\|_s)$ and $T$ such that
\begin{equation}\label{time-contract}
cT^{1/2}\big (\tres (u,v)\tres+\tres (u,v)\tres^2\big) <\frac12
\end{equation}
we can show that the map $(\Phi (u), \Psi(v))$ applies the ball 
$$
X^T_a=\{ (u,v)\in C([0,{T_0}]; H^{s+1/2})\times C([0, T]; H^{s}) \;: \:\tres(\Phi(u),\Psi(v))\tres\le a\} 
$$ 
into itself.

The same argument described above show that $(\Phi(u),\Psi(v))$ is a contraction in $X^T_a$ and so there is a unique
solution of the IVP \eqref{IVP}
\begin{equation}
(u,v)\in C([0,T]; H^{s+1/2}(\R))\times C([0, T]; H^{s}(\R)).
\end{equation}

By uniqueness the previous argument gives us a solution $(u(t), v(t))$ defined by the class \eqref{spaces-1}-\eqref{spaces-2} of the
integral equations.
\begin{equation}\label{solution}
\begin{split}
u(t) &= S(t)u_0 +\int_0^t S(t-t') (uv)(t')\,dt' +\int_0^t S(t-t') |u|^2u(t')\,dt',\\
v(t) &= V(t) v_0 -\int_0^t V(t-t') v\p_xv(t')\,dt' + \int_0^t V(t-t') \p_x(|u|^2)(t')\, dt'.
\end{split}
\end{equation}

\medskip

Next we prove the persistence property in weighted spaces. For simplicity we will take $s=\frac34+$ in the following.

We consider
\begin{equation}
u_0\in H^{\frac54+}(\R)\cup L^2(|x|^{\frac{5}{2}^{+}}dx) \text{\hskip10pt and\hskip10pt}  v_0\in H^{\frac34+}(\R)\cup L^2(|x|^{\frac{3}{4}^{+}}dx)
\end{equation}
and introduce the notation
\begin{equation}
\begin{split}
\mu_3^{T_0}(u)&=\mu_1^{T_0} (u)+\underset{[0,{T_0}]}{\sup} \||x|^{\frac{5}{4}^{+}}u(t)\|_{L^2}\\
\mu_4^{T_0}(v) &=\mu_2^{T_0} (v)+\underset{[0,{T_0}]}{\sup} \||x|^{\frac{3}{8}^{+}}v(t)\|_{L^2}
\end{split}
\end{equation}
for some $T_0\in (0, T)$ to be determined below.

\medskip

Thus, applying formula \eqref{lwp-12} to $u$ in \eqref{solution}, we have
\begin{equation}\label{lwp-14}
\begin{split}
&\| |x|^{\frac54^{+}} (u)(t)\|_{L^2} \\
&\le  c  \| |x|^{\frac54^{+}} S(t)u_0\|_{L^2}+\| |x|^{\frac54^{+}} S(t)\int_0^t S(t')(uv+|u|^2u)(t')\,dt'\|_{L^2} \\
&\le  \||x|^{\frac54^{+}}u_0\|_{L^2}+ c(1+|t|)\,(\|u_0\|_{L^2}+\|D^{\frac54^{+}}_x u_0\|_{L^2})\\
&\;\;\; + c \int_0^{T_0} \||x|^{\frac54^{+}}(uv+|u|^2u)(t)\|_{L^2}\,dt \\
&\;\;\;+ c\,(1+{T_0})\int_0^{T_0} \|(uv+|u|^2u)(t)\|_{L^2}\,dt\\
&\;\;\;+c \,(1+{T_0})  \int_0^{T_0} \| D^{\frac54^{+}}_x(uv+|u|^2u)(t)\|_{L^2}\,dt\\
&\le \||x|^{\frac54^{+}}u_0\|_{L^2}+ c(1+{T_0})\|u_0\|_{\frac{5}{4}^{+}}+ A_1+A_2+A_3.
\end{split}
\end{equation}

Next we estimate $A_i$, $i=1,2,3$.  Holder's inequality and Sobolev lemma lead to
\begin{equation}\label{lwp-15}
\begin{split}
A_1 &\le \int_0^{{T_0}} (\| |x|^{\frac54^{+}}u v(t)\|_{L^2} +\||x|^{\frac54^{+}}u|u|^2(t)\|_{L^2})\,dt\\
&\le c\,{T_0} \big( \| |x|^{\frac{5}{4}^{+}}u\|_{L^{\infty}_{T_0}L^2_x} \underset{[0,{T_0}]}{\sup}\|v(t)\|_{\frac34^{+}}
+\| |x|^{\frac{5}{4}^{+}}u(t)\|_{L^{\infty}_{T_0}L^2_x} \underset{[0,{T_0}]}{\sup}\|u(t)\|_{\frac54^{+}}^2\big) \\
&\le cT_0\,\big(\mu_3^{T_0}(u)\mu_2^{T_0}(v)+ \mu_3^{T_0}(u)(\mu_1^{T_0}(u))^2\big).
\end{split}
\end{equation}

Holder's inequality and Sobolev lemma yield
\begin{equation}\label{lwp-16}
\begin{split}
A_2 &\le  c\,(1+{T_0})\,{T_0} \big( \underset{[0,{T_0}]}{\sup}\|u(t)\|_{\frac54^{+}} \underset{[0,{T_0}]}{\sup}\| v(t)\|_{\frac34^{+}} + \underset{[0,{T_0}]}{\sup}\|u(t)\|_{\frac54^{+}}^3 \big)\\
&\le  c\,(1+{T_0})\,{T_0} \big(\mu_1^{T_0}(u)\mu_2^{T_0}(v)+(\mu_1^{T_0}(u))^3\big).
\end{split}
\end{equation}

Applying Sobolev spaces properties we obtain
\begin{equation}\label{lwp-17}
\begin{split}
A_3 &\le  c\,(1+{T_0})\,\int_0^{T_0} \big(\|D^{\frac58^{+}}_x (uv)(t)\|_{L^2} + \|D^{\frac58^{+}}_x (u|u|^2)(t)\|_{L^2} \big)\,dt\\
&\le c\,(1+{T_0})\,{T_0} \big( \underset{[0,{T_0}]}{\sup}\|u(t)\|_{{\frac54}^{+}} \underset{[0,{T_0}]}{\sup}\| v(t)\|_{{\frac34}^{+}} + \underset{[0,{T_0}]}{\sup}\|u(t)\|_{{\frac54}^{+}}^3 \big)\\
&\le  c\,(1+{T_0})\,{T_0} \big(\mu_1^{T_0}(u)\mu_2^{T_0}(v)+(\mu_1^{T_0}(u))^3\big).
\end{split}
\end{equation}


Now we estimate $\| |x|^{\frac38^{+}} v(t)\|_{L^2}$. Applying formula \eqref{kdv_persistencia} to $v$ in \eqref{solution} we get
\begin{equation}\label{lwp-18}
\begin{split}
\| |x|^{\frac38^{+}} v(t)\|_{L^2} &\le \| |x|^{\frac38^{+}}v_0\|_{L^2} +c(1+{T_0})\,\big(\|v_0\|_{L^2}+\|D^{\frac34^{+}}_x v_0\|_{L^2}\big)\\
&\;\;\;+  \|\int_0^t V(t-t')|x|^{\frac38^{+}} v\partial_x v \,dt' \|_{L^2}\\
&\hskip10pt + \|\int_0^t V(t-t')|x|^{\frac38^{+}}\partial_x |u|^2\,dt'\|_{L^2}\\
&\;\;\;+c\,(1+{T_0})\,\int_0^{T_0}\big(\| v\p_x v\|_{L^2}+\| \p_x |u|^2\|_{L^2}\big)\,dt\\
&\;\;\;+c\,(1+{T_0})\,\int_0^{T_0}\|D^{\frac34^{+}}_x (v\p_xv)\|_{L^2}\,dt \\
&\;\;\;+c\,(1+{T_0})\,\int_0^{T_0} \|D^{\frac34^{+}}_x \p_x(|u|^2)\|_{L^2}\,dt.
\end{split}
\end{equation}

The last three terms above were previously estimated. We only need to bound the third and fourth term on the
right hand side of \eqref{lwp-18}. 

\medskip

Minkowski's inequality, group properties and H\"older's inequality yield
\begin{equation}\label{lwp-19}
\begin{split}
\|\int_0^t V(t-t') |x|^{\frac38^{+}}\,v\partial_x v(t')dt'\|_{L^2}
&\le \int_0^{T_0} \| |x|^{\frac38^{+}}v\|_{L^2}\|\p_xv\|_{L^{\infty}_x}\,dt\\
&\le cT^{3/4}_0 \sup_{[0,T_0]} \| |x|^{\frac38^{+}}v(t)\|_{L^2}\|\p_xv\|_{L^4_{T_0}L^{\infty}_x}\\
&\le cT^{3/4}_0 \mu_4^{T_0}(v) \mu_2^{T_0}(v).
\end{split}
\end{equation}

Similarly, we obtain
\begin{equation}\label{lwp-20}
\begin{split}
\|\int_0^t &V(t-t') |x|^{\frac38^{+}}\p_x(|u|^2)(t')dt'\|_{L^2_x} \\
&\le \int_0^{T_0} (\||x|^{\frac38^{+}}\bar{u}\|_{L^2_x} \|\p_x u\|_{L^{\infty}}+ \||x|^{\frac38^{+}}u\|_{L^2_x}\|\p_x \bar{u}\|_{L^{\infty}_x})\,dt\\
&\le \int_0^{T_0} (\|\ji x\jd^{\frac54^{+}}\bar{u}\| \|\p_x u\|_{L^{\infty}}+ \|\ji x\jd^{\frac54^{+}}u\|_{L^2_x}\|\p_x \bar{u}\|_{L^{\infty}_x})\,dt\\
&\le c T^{3/4}_0\sup_{[0,T_0]} \|\ji x\jd^{\frac54^{+}}u(t)\|_{L^2} \|\p_x u\|_{L^4_TL^{\infty}_x}\\
&\le cT^{3/4}_0 \mu_3^{T_0}(u) \mu_1^{T_0}(u).
\end{split}
\end{equation}

Gathering the information in \eqref{lwp-14}-\eqref{lwp-17} we get that
\begin{equation}\label{sol-wu}
\begin{split}
 \||x|^{\frac54^{+}}u(t)\|_{L^2}&\le  c\||x|^{\frac54^{+}}u_0\|_{L^2}+ c(1+T_0)\,\|u_0\|_{\frac54^{+}}\\
& \hskip10pt +cT_0\,\big(\mu_3^{T_0}(u)\mu_2^{T_0}(v)+ \mu_3^{T_0}(u)(\mu_1^{T_0}(u))^2\big)\\
& \hskip10pt  + c\,(1+{T_0})\,{T_0} \big(\mu_1^{T_0}(u)\mu_2^{T_0}(v)+(\mu_1^{T_0}(u))^3\big).
\end{split}
\end{equation}

On the other hand,  from \eqref{lwp-18}-\eqref{lwp-20} we deduce that
\begin{equation}\label{sol-wv}
\begin{split}
\| |x|^{\frac38^{+}} v(t)\|_{L^2} &\le c\| |x|^{\frac38^{+}}v_0\|_{L^2} +c(1+{T_0})\,\|v_0\|_{3/4^{+}}\\
& \hskip10pt + cT^{3/4}_0 \big(\mu_4^{T_0}(v) \mu_2^{T_0}(v)+ \mu_3^{T_0}(u) \mu_1^{T_0}(u)\big)\\
& \hskip 10pt  +cT^{1/2}_0\big((\mu_2^T(v))^2 +(\mu_1^T(u))^2\big).
\end{split}
\end{equation}

Taking $T_0\in (0,T)$ such that \eqref{time-contract} holds we obtain
\begin{equation}\label{sol-weighted}
\begin{split}
 \||x|^{\frac54^{+}}u(t)\|_{L^2}+\| |x|^{\frac38^{+}} v(t)\|_{L^2}
 &\le  c\||x|^{\frac54^{+}}u_0\|_{L^2}+ c(1+T_0)\,\|u_0\|_{\frac54^{+}}\\
 & \hskip10pt +c\| |x|^{\frac38^{+}}v_0\|_{L^2} +c(1+T_0)\,\|v_0\|_{\frac34^{+}}\\
& \hskip10pt +2c(1+T_0) T_0^{1/2} (\|u_0\|_{s+1/2}+\|v_0\|_s)\\
&  \hskip10pt+ 2c\, T^{1/2}_0(\|u_0\|_{\frac54^{+}}+\|v_0\|_{\frac34^{+}}).
\end{split}
\end{equation}
This  basically completes the proof.
\end{proof}



\bigskip

\section{Proof of the Main Theorem \ref{MT2}}

\medskip

The following proof is built upon the linear analysis appearing in Section \ref{section_ID}.

\medskip

Consider the IVP \eqref{IVP} associated to the Schr\"odinger-Korteweg-de Vries system, 
with initial data
\[
u_0\in C^\infty(\mathbb{R})\cap H^{2^-}(\mathbb{R})\cap L^2(\langle x\rangle ^{4^-}dx)\cap L^\infty(\mathbb{R}),
\]
and \[
v_0\in C^{\infty}(\mathbb{R})\cap H^{3/2^-}(\mathbb{R})\cap L^2(\langle x\rangle^{3/2^-}dx)\cap L^\infty (\mathbb{R}),
\] 
constructed in Lemmas \ref{dbu_lsch} and Lemma \ref{dbu_lkdv} respectively, choosing both parameters such that they develop dispersive blow-up for the linear equations at the same time $t^*$ small enough. 

\medskip

As we proved in the previous section we have a solution of the IVP \eqref{IVP} given by
\begin{align}\label{u_sol}
u(x,t)=S(t)u_0+\int_0^t S(t-t')(uv)(t')dt'+\int_0^tS(t-t')(\vert u\vert^2u)(t')dt',
\end{align}
and 
\begin{equation}\label{v_sol}
\begin{split}
v(x,t) &=V(t)v_0+\int_0^tV(t-t')\partial_x(v^2)(t')dt'
+\int_0^tV(t-t')(\partial_x\vert u\vert^2)(t')dt'.
\end{split}
\end{equation}
If the integral terms in \eqref{u_sol} and \eqref{v_sol} are $C^{1,\frac{1}{2}+\varepsilon}_x(\mathbb{R})$ and $C^1_x(\mathbb{R})$ functions for all $t\in[0,T]$ respectively, then the desired result will follow from what we already known about $S(t)u_0$ and $V(t)v_0$ in Section \ref{section_ID}. To do this we divide the analysis in two cases, the inhomogeneous terms at Schr\"odinger equation level and the other ones at KdV equation level. 

\medskip 

The following two lemmas are sufficient to complete the proof of Theorem \ref{MT2}.

\begin{lem}\label{DBU_Schr}
Let $s>3/4 $ and consider an initial data \[
(u_0,v_0)\in H^{s+\frac{1}{2}}(\mathbb{R})\cap L^2(\vert x\vert^{2s+1}dx)\times H^s(\mathbb{R})\cap L^2(\vert x\vert^{s}dx).
\]
 Let $(u(t),v(t))$ be the corresponding solution for the IVP \eqref{IVP} given by Theorem \ref{MT1},
\begin{align*}
u(x,t)&=S(t)u_0+\int_0^t S(t-t')(uv)(t')dt'+\int_0^tS(t-t')(\vert u\vert^2u)(t')dt'
\\ &=:S(t)u_0+\mathrm{I}(x,t),
\end{align*}
then $\mathrm{I}\in C([0,T];H^{s+\frac{3}{4}}(\mathbb{R}))$.
\end{lem}
In other words, the integral term $\mathrm{I}$ is smoother than the free propagator $e^{it\Delta}u_0$ by a quarter of derivative. In particular, this implies that for initial data as at the beginning of this section, the integral term $\mathrm{I}\in C([0,T]: C^{1,\frac{3}{4}^-}(\mathbb{R}))$.

\medskip

\begin{lem}\label{DBU_KdV}
Let $s>7/6$ and consider an initial data \[
(u_0,v_0)\in H^{s+\frac{1}{2}}(\mathbb{R})\cap L^2(\vert x\vert^{2s+1}dx)\times H^s(\mathbb{R})\cap L^2(\vert x\vert^{s}dx).
\]
Let $(u(t),v(t))$ be the corresponding solution for the IVP \eqref{IVP} given by Theorem \ref{MT1},
\begin{align*}
v(x,t)&=V(t)v_0+\int_0^t V(t-t')\partial_x(v^2)(t')dt'+\int_0^tV(t-t')(\partial_x\vert u\vert^2)(t')dt'
\\ & =:V(t)v_0+\mathrm{II}(x,t),
\end{align*}
then $\mathrm{II}\in C([0,T];H^{s+\frac{1}{6}}(\mathbb{R}))$.
\end{lem}
The lemma affirms that the integral term $\mathrm{II}$ is smoother than the free propagator $V(t)v_0$ by a sixth derivative. In particular, this implies that for initial data $v_0$ as above, the integral term $\mathrm{II}\in C([0,T]: C^1(\mathbb{R}))$.

\medskip

\begin{proof}[Proof of Lemma \ref{DBU_Schr}]
First of all, recall that the local well-posedness Theorem \ref{MT1} guarantees the existence of the solution \[ 
u\in C([0,T]:H^{s+\frac{1}{2}}(\mathbb{R})\cap L^2(\vert x\vert^{2s+1}dx)), \quad v\in C([0,T]:H^s(\mathbb{R})\cap L^2(\vert x\vert^{s}dx)).
\] 
Now, let us divide the analysis in two steps. First, define
\[
u_1(t):=\int_0^t S(t-t')(uv)(t')dt'.
\]
We shall show that $u_1(t)\in H^{s+\frac{3}{4}}(\mathbb{R})$ for all $t\in[0,T]$. In fact, by \eqref{dual-kato-s-schr1} we have 
\begin{equation}\label{schr_est1}
\begin{split}
\Vert D^{s+\frac{3}{4}}_xu_1\Vert_{L^2_x}& \leq \big\Vert D^{1/2}_x \int_0^tS(t-t')D^{s+\frac{1}{4}}_x(uv)(t')dt'\big\Vert_{L^2} \\
&\leq \Vert D^{s+\frac{1}{4}}_x(uv)\Vert_{L^1_xL^2_T}\\
& \leq \Vert vD^{s+\frac{1}{4}}_xu\Vert_{L^1_xL^2_T}+\Vert uD^{s+\frac{1}{4}}_xv\Vert_{L^1_xL^2_T}+E_1, 
\end{split}
\end{equation}
Let us estimate each of these terms. Using H\"older's inequality we can bound the first term of \eqref{schr_est1} by \begin{align*}
\Vert vD^{s+\frac{1}{4}}_xu\Vert_{L^1_xL^2_T}&\leq c\Vert v \Vert_{L^2_xL^\infty_T}\Vert D^{s+\frac{1}{4}}_x u\Vert_{L^2_xL^2_T}
\\ & \leq cT^{1/2}\Vert v\Vert_{L^2_xL^\infty_T}\Vert D^{s+\frac{1}{4}}_xu\Vert_{L^\infty_TL^2_x}
\\ & <\infty.
\end{align*}
On the other hand, the second term of \eqref{schr_est1} can be bound by 
\begin{align*}
\Vert uD^{s+\frac{1}{4}}_xv\Vert_{L^1_xL^2_T}&\leq c\Vert \langle x\rangle^{1/2}uD^{s+\frac{1}{4}}_xv\Vert_{L^2_TL^2_x}
\\ & \leq cT^{3/4}\Vert \langle x\rangle^{s+\frac12-\varepsilon}u\Vert_{L^\infty_TL^2_x}\Vert D^{s+\frac{1}{4}}_xv\Vert_{L^4_TL^\infty_x}
\\ & <\infty.
\end{align*}
Now to estimate $E_1$ we shall employ commutator estimates and interpolated norms of the previous terms. For the sake of completeness
we sketch the proof. 
\begin{equation*}
\begin{split}
E_1&=\|D^a_x(u\p_x v)-uD^a_x\p_x v-\p_xvD^a_x u\|_{L^1_xL^2_T}\\
&\hskip10pt+\|D^a_x(v\p_x u)-vD^a_x\p_x u-\p_xuD^a_x v\|_{L^1_xL^2_T}\\
&\hskip10pt+\|\p_xuD^a_x v\|_{L^1_xL^2_T}+\|\p_xvD^a_x u\|_{L^1_xL^2_T}
\end{split}
\end{equation*} 
where $a\in (0,1)$ is such that $s+\frac14=1+a$.

Applying the Leibnitz rule \eqref{LR-L1L2} and H\"older's inequality it follows that
\begin{equation*}
\begin{split}
E_1 & \le  c\, \|\langle x\rangle^{1/4^+}\p_xu\|_{L^{4}_TL^{4}_x}\|\langle x\rangle^{1/4}D^a_x v\|_{L^4_TL^{4}_x} +c\, \|\p_x v\|_{L^{20/3}_xL^5_T}\|D^a_xu\|_{L^{20/17}_xL^{10/3}_T}\\
&= E_{1,1}+E_{2,2}.
\end{split}
\end{equation*}

Let us first estimate the term $E_{1,1}$. To estimate  $\|\langle x\rangle ^{1/4^+}\p_xu\|_{L^{4}_TL^{4}_x}$ we will use Strichartz estimate \eqref{weighted-sch} combined with the weighted estimate \eqref{lwp-14}. Indeed, H\"older inequality and interpolation inequalities in lemma \ref{nahas-ponce} give us 
\begin{align*}
\Vert \langle x\rangle^{1/4^+}\partial_xu\Vert_{L^4_TL^4_x}&\leq cT^{1/8}\Vert \langle x\rangle^{1/4^+}\partial_xu\Vert_{L^8_TL^4_x}
\\ & \leq cT^{1/8}\Vert \langle x\rangle^{1/4^+}\partial_x u\Vert_{L^\infty_TL^2_x}+cT^{1/8}(1+T)\Vert u\Vert_{L^\infty_TH^{s+1/2}_x}
\\ & \leq cT^{1/8}\Vert u\Vert_{L^\infty_TH^{s+1/2}_x}^{2/(2s+1)} \Vert \langle x\rangle^{\frac{2s+1}{4(2s-1)}^+} u\Vert_{L^\infty_TL^2_x}^{(2s-1)/(2s+1)}+
\\ & \quad + cT^{1/8}(1+T)\Vert u\Vert_{L^\infty_TH^{s+1/2}_x},
\end{align*}
which is finite thanks to the fact $\tfrac{2s+1}{4(2s-1)}<s+\tfrac{1}{2}$.

Let us now estimate $\|\langle x\rangle ^{1/4}D^a_xv\|_{L^{4}_TL^{4}_x}$. Using Sobolev's embedding and the interpolation inequalities in Lemma \ref{nahas-ponce} we obtain
\begin{align*}
\Vert \langle x\rangle^{1/4}D^a_xv\Vert_{L^4_TL^4_x}&\leq cT^{1/4}\Vert J^{1/4}\langle x\rangle^{1/4}D^a_xv\Vert_{L^\infty_TL^2_x}
\\ &\leq cT^{1/4}\Vert v\Vert_{L^\infty_TH^s_x}^{1/3}\Vert \langle x\rangle^{3/8}D^a_xv\Vert_{L^\infty_TL^2_x}^{2/3}
\\ & \leq cT^{1/4}\Vert v\Vert_{L^\infty_TH^{s}_x}^{(2s-1)/2s}\Vert \langle x\rangle^{s/2}v\Vert_{L^\infty_TL^2_x}^{1/2s}
\end{align*}
which is finite.

To estimate  $\|\p_x v\|_{L^{20/3}_xL^5_T}$  we employ the linear estimate \eqref{inter-2} and a similar argument to show that
the solution $v$ is in $H^s$, $s>3/4$. To bound $\|D^a_xu\|_{L^{20/17}_xL^{10/3}_T}$, using the same ideas as in the previous estimation we obtain
\begin{align*}
\Vert D^a_xu\Vert_{L^{20/17}_xL^{10/3}_T}&\leq c\Vert \langle x\rangle^{11/20^+}D^a_xu\Vert_{L^{10/3}_TL^{10/3}_x}
\\ & \leq  cT^{1/5}\Vert \langle x\rangle^{11/20^+}D^a_xu\Vert_{L^{10}_TL^{10/3}_x}
\\ & \leq cT^{1/5}\Vert \langle x\rangle^{11/20^+}D^a_xu\Vert_{L^{\infty}_TL^{2}_x}+cT^{1/5}(1+T)\Vert u\Vert_{L^\infty_T H ^{s+1/2}_x}
\\ & \leq cT^{1/5}\Vert u\Vert_{L^\infty_T H^{s+1/2}_x}^{(4s-3)/(4s+2)}\Vert \langle x\rangle^{11(2s+1)/50^+}u\Vert_{L^\infty_TL^2_x}^{5/(4s+2)}+
\\ &\quad +cT^{1/5}(1+T)\Vert u\Vert_{L^\infty_TH^{s+1/2}_x}
\end{align*}
which is finite thanks to the fact $\tfrac{11(2s+1)}{50}<s+\tfrac{1}{2}$ and
therefore we have \[
u_1(t)\in H^{s+\frac{3}{4}}(\mathbb{R}) \quad\hbox{for all } \,t\in[0,T].
\]
Now, let us consider the second integral term of the solution $u(x,t)$: \[
u_2(t):=\int_0^tS(t-t')(\vert u\vert^2u)(t')dt'.
\]
We shall show that $u_2(t)\in H^{s+\frac{3}{4}}(\mathbb{R})$ for all $t\in[0,T]$. In fact, by the dual version of Kato's smoothing effect \eqref{dual-kato-s-kdv} we have
\begin{equation}
\begin{split}
\Vert D^{s+\frac{3}{4}}_xu_2\Vert_{L^2_x}&\leq c \Vert D^{s+\frac{1}{4}}_x(\vert u\vert^2u)\Vert_{L^1_xL^2_T}
\\ & \leq c\Vert u \Vert_{L^4_xL^\infty_T}^2\Vert D^{s+\frac{1}{4}}_xu\Vert_{L^\infty_xL^2_T} +E_2,
\\ & <\infty.
\end{split}
\end{equation}
where, again, the terms in $E_2$ are easy to control by considering the commutator estimates (see \cite{KPV}) and the interpolated norms of the previous terms, so we omit the details. Above we have used the estimate \eqref{sch-mfn-4} applied to \eqref{u_sol} and then the arguments
in \eqref{lwp-1} and \eqref{lwp-2} to bound $\Vert u \Vert_{L^4_xL^\infty_T}$. Thus we conclude that $\mathrm{I}\in C([0,T];H^{s+\frac{3}{4}}(\R))$.

\end{proof}

\begin{proof}[Proof of Lemma \ref{DBU_KdV}]
First of all, recall that the local well-posedness Theorem \ref{MT1} guarantees the existence of the solution \[ 
u\in C([0,T]:H^{s+\frac{1}{2}}(\mathbb{R})\cap L^2(\vert x\vert^{2s+1}dx)), \quad v\in C([0,T]:H^s(\mathbb{R})\cap L^2(\vert x\vert^{s}dx)).
\] 
Now, let us divide the analysis in two steps. First, define \[
v_1(t):=\int_0^t V(t-t')\,\partial_x(v^2)(t')dt'\,\in\,C^1(\mathbb{R}).
\]
We shall show that $v_1(t)\in H^{s+\frac{1}{6}}(\mathbb{R})$ for all $t\in[0,T]$.
In fact, using the smoothing Kato effect \eqref{dual-kato-s-kdv}, we obtain \begin{align*}
\sup_{0\leq t\leq T}\Vert D^{s+1/6}_x\int_0^tV(t-t')(v\partial_xv)(t')dt'\Vert_2&\leq \Vert D^{s-5/6}_x(v\partial_x v)\Vert_{L_x^1L_T^2}
\\ & \leq \Vert v\Vert_{L_x^{6/5}L_T^3}\Vert D^{s+1/6}_xv\Vert_{L^6_xL^6_T}+E_1,
\end{align*}
where $E_1$ are easy to control by considering the commutator estimates (see \cite{KPV}) and interpolated norms of the previous terms to be considered below, so we omit this proof. Now, from Strichartz estimates \eqref{str-kdv} with $p=q=6$, $\theta=\tfrac{2}{3}$ and $\alpha=\tfrac{1}{2}$ we obtain: \[
\Vert D^{s+1/6}_xv\Vert_{L_x^6L_T^6}<\infty.
\]
On the other hand, using \eqref{nahas-ponce-ineq} in Lemma \ref{nahas-ponce} we deduce: \begin{align*}
\Vert v\Vert_{L^{6/5}_xL^3_T}&\leq c\Vert \langle x\rangle^{1/2^+}v\Vert_{L^3_TL^3_x}
\\ & \leq cT^{1/3}\Vert \langle x\rangle^{1/2^+}v\Vert_{L_T^\infty L_x^3}
\\ & \leq cT^{1/3}\Vert J^{1/6}\big(\langle x\rangle^{1/2^+}v\big)\Vert_{L_T^\infty L_x^2} 
\\ & \leq cT^{1/3}\Vert J^{s}v\Vert_{L^\infty_TL^2_x}^{1-\gamma}\Vert \langle x\rangle ^{\frac{s}{2}^-}v\Vert_{L^\infty_TL^2_x}^\gamma,
\end{align*}
with $\gamma$ such that $\tfrac{s\gamma}{2}^-=\tfrac{1}{2}^+$, i.e. $\gamma> \tfrac{1}{s}$, and such that $(1-\gamma)s>1/6$. Note that the last inequality imposes the restriction $s>\tfrac{7}{6}$. Thus we have $v_1(t)\in H^{s+\frac{1}{6}}(\mathbb{R})$ for all $t\in[0,T]$, which concludes the demonstration of the first step.

\medskip

Now, let us consider the second integral term of the solution $v(t,x)$:
\[
v_2(t):=\int_0^t V(t-t')\big(\partial_x\vert u \vert^2\big)(t')dt'\,\in\,C^1(\mathbb{R})
\]
We shall show that $v_2(t)\in H^{s+\frac{1}{6}}(\mathbb{R})$ for all $t\in[0,T]$. For this, we use the inhomogeneous smoothing Kato effect \eqref{dual-kato-s-kdv}, thus we obtain: 
\begin{align*}\sup_{0\leq t \leq T}\big\Vert D^{s+\frac{1}{6}}_x\int_0^tV(t-t')\big(\partial_x\vert u \vert^2 \big)dt'\big\Vert_2&\leq c\big\Vert D^{s+\frac{1}{6}}_x\big(\vert u\vert ^2 \big)\big\Vert_{L^1_xL^2_T}
\\ & \leq c\Vert u\Vert_{L^2_xL^\infty_T}\Vert D^{s+\frac{1}{6}}_xu\Vert_{L^2_TL^2_x}+E_2
\\ & \leq c T^{1/2}\Vert u\Vert_{L^2_xL^\infty_T}\Vert D^{s+\frac{1}{6}}_xu\Vert_{L^\infty_TL^2_x}+E_2,
\end{align*}
where $\Vert u\Vert_{L^2_xL^\infty_T}<\infty$ due to Theorem \ref{MT1} and the terms in $E_2$ are easy to control by considering the commutator estimates and the interpolated norms of the previous terms.

This concludes the estimates for the solution $v$. 

\medskip

Therefore we have shown that the Duhamel terms associated to our solutions are smoother that the corresponding linear associated solutions.
In consequence, if there is a point singularity it has to be provided by the linear solution.


\end{proof}




\medskip

\begin{thebibliography}{99}

\medskip


\bibitem{BeOgPo} Bekiranov, Daniella; Ogawa, Takayoshi; Ponce, Gustavo, \emph{Interaction equations for short and long dispersive waves}. J. Funct. Anal. 158 (1998), no. 2, 357–388.

\bibitem{BeOgPo2} Bekiranov, Daniella; Ogawa, Takayoshi; Ponce, Gustavo, \emph{Weak solvability and well-posedness of a coupled Schrödinger-Korteweg de Vries equation for capillary-gravity wave interactions}. Proc. Amer. Math. Soc. 125 (1997), no. 10, 2907–2919. 

\bibitem{BeBu} E. S. Benilov and S. P. Burtsev, \emph{To the integrability of the equations describing the Langmuir-wave-ion-acoustic-wave interaction}. Phys. Let., 98A (1983), 256–258. MR0720816 (85f:76120) 

\bibitem{BeBoMa} T. B. Benjamin, J. L. Bona, J. J. Mahony, \emph{Model equations for long waves in nonlinear, dispersive media}, Philos. Trans. R. Soc. Lond. A272(1972)47–78.

\bibitem{BoSa} Bona, J. L.; Saut, J.-C., \emph{Dispersive blowup of solutions of generalized Korteweg-de Vries equations}. J. Differential Equations 103 (1993), no. 1, 3–57.

\bibitem{BoPoSaSp} Bona, J. L.; Ponce, G.; Saut, J.-C.; Sparber, C., \emph{Dispersive blow-up for nonlinear Schr\"odinger equations revisited}. 
J. Math. Pures Appl. (9) 102 (2014), no. 4, 782–811. 

\bibitem{BoSa2} Bona, Jerry L.; Saut, Jean-Claude, \emph{Dispersive blow-up II. Schrödinger-type equations, optical and oceanic rogue waves}. Chin. Ann. Math. Ser. B 31 (2010), no. 6, 793–818. 

\bibitem{Bo} J. Bourgain, \emph{Fourier transform restriction phenomena for certain lattice subsets and applications to nonlinear evolution equations}, Geometric and Functional Anal., 3 (1993), 107–156, 209–262. MR1209299 (95d:35160a); MR1215780 (95d:35160b)


\bibitem{CoLi} Corcho, A. J.; Linares, F., \emph{Well-posedness for the Schr\"odinger-Korteweg-de Vries system}, Trans. Amer. Math. Soc. 359 (2007), no. 9, 4089–4106. 

\bibitem{FoLiPo} Fonseca, G., Linares, F., Ponce, G.: \emph{On persistence properties in fractional weighted spaces}. Proc. Am. Math. Soc. 143, 5353–5367 (2015)

\bibitem{FuOi} M. Funakoshi and M. Oikawa, \emph{The resonant interaction between a long internal gravity wave and a surface gravity wave packet}, J. Phys. Soc. Japan, 52 (1983), 1982–1995. MR0710730
(84k:76030)

\bibitem{ZiYu} Zihua, Guo; Wang, Yuzhao, \emph{On the well-posedness of the Schrödinger-Korteweg-de Vries system}, J. Differential Equations 249 (2010), no. 10, 2500–2520.

\bibitem{HoIkMiNi} H. Hojo, H. Ikezi, K. Mima and K. Nishikawa, \emph{Coupled nonlinear electron-plasma and ionacoustic waves}, Phys. Rev. Lett., 33 (1974), 148–151.

\bibitem{KaKaSu} T. Kakutani, T. Kawahara and N. Sugimoto, \emph{Nonlinear interaction between short and long capillary-gravity waves}, J. Phys. Soc. Japan, 39 (1975), 1379–1386.

\bibitem{K} Kato, T., \emph{On the Cauchy problem for the (generalized) Korteweg-de Vries equation}. Studies in applied mathematics, 93–128, 
Adv. Math. Suppl. Stud., 8, Academic Press, New York, 1983.

\bibitem{KPV} Kenig, C.; Ponce, G.; Vega, L., \emph{Well-posedness and scattering results for the generalized Korteweg-de Vries equation via the contraction principle}, Comm. Pure Appl. Math. 46 (1993), no. 4, 527–620.

\bibitem{KPV2} Kenig, C.; Ponce, G.; Vega, L., \emph{A bilinear estimate with applications to the KdV equation}, J. Amer. Math. Soc., 9 (1996), 573–603. MR1329387 (96k:35159) .

\bibitem{KPV3} C. E. Kenig, G. Ponce and L. Vega, \emph{
The Cauchy problem for the Korteweg-de Vries equation
in Sobolev spaces of negative indices}, Duke Math. J., 71 (1993), 1–21. MR1230283
(94g:35196)

\bibitem{LiPo} F. Linares and G. Ponce, \emph{Introduction to Nonlinear Dispersive Equations}. Universitext. Springer, New York, 2009


\bibitem{LiPoSm} Linares, F.; Ponce, G.; Smith, D., \emph{On the regularity of solutions to a class of nonlinear dispersive equations}, 
Math. Ann. 369 (2017), no. 1-2, 797–837. 

\bibitem{LiSc} Linares, F.; Scialom, M., \emph{On the smoothing properties of solutions to the modified Korteweg-de Vries equation}, J. Differential Equations 106 (1993), no. 1, 141–154.
 
 \bibitem{Li}  Linares, F. {\em A higher order modified Korteweg-de Vries equation}, Mat. Apl. Comput. 14 (1995), no. 3, 253--267.

\bibitem{NP} Nahas J. and Ponce G., \emph{On the persistent properties of solutions to semi-linear Schr\"odinger equation}, Comm. Partial Differential Equations 34 (2009), no. 10-12, 1208–1227.

\bibitem{RoVi} K. M. Rogers, P. Villarroya, \emph{Global estimates for the Schr\"odinger maximal operator}, Ann. Acad. Sci. Fenn. 32 (2007) 425–435. 

\bibitem{SaYa} J. Satsuma and N. Yajima, \emph{Soliton solutions in a diatomic lattice system}, Progr. Theor.
Phys., 62 (1979), 370–378.


\bibitem{Ts} Y. Tsutsumi, \emph{$L^2$-solutions for nonlinear Schr\"odinger equations and nonlinear groups}, Ekvacioj, 30 (1987), 115–125. MR0915266 (89c:35143)

\bibitem{Ve2} Vega, L., \emph{El multiplicador de Schr\"odinger, la funcion maximal y los operadores de restriccrci\'on}, Doctoral Thesis, Universidad Autonoma de Madrid, 1988.

\bibitem{Ve} L. Vega, \emph{Schr\"odinger equations: pointwise convergence to the initial data}, Proc. Am. Math. Soc. 102 (1988) 874–878.


\end{thebibliography}
\end{document}